\documentclass[12pt]{article}
\usepackage[affil-it]{authblk}
\usepackage{amsmath} 
\usepackage{amssymb} 
\usepackage{makeidx} 
\usepackage{graphicx} 
\usepackage{enumerate}
\usepackage[utf8]{inputenc}
\usepackage{amsthm}

\newtheorem{prop}{Proposition}
\newtheorem{teor}[prop]{Theorem}

\DeclareMathOperator{\PG}{PG}

\newcommand{\calc}{{\cal C}}
\newcommand{\cald}{{\cal D}}
\newcommand{\cale}{{\cal E}}
\newcommand{\calf}{{\cal F}}
\newcommand{\calq}{{\cal Q}}
\newcommand{\calr}{{\cal R}}
\newcommand{\cals}{{\cal S}}
\newcommand{\FF}{\mbox{$\mathbb F$}}
\newcommand{\PP}{\mbox{$\mathbb P$}}

\begin{document}

\title{Geometry of the inversion in a finite field and partitions of $\PG(2^k-1,q)$ in normal rational curves}

\date{\today}
\author{Michel Lavrauw - Corrado Zanella}
\affil{Dipartimento di Tecnica e Gestione dei Sistemi Industriali, 36100 Vicenza, Italy}

\maketitle

\begin{abstract}
  Let $L=\mathbb F_{q^n}$ be a finite field and let $F=\mathbb F_q$ be a subfield of $L$.
  Consider $L$ as a vector space over $F$ and the associated projective space
  that is isomorphic to $\PG(n-1,q)$.
  The properties of the projective mapping induced by $x\mapsto x^{-1}$
  have been studied in \cite{Cs13,Fa02,Ha83,He85,Bu95}, where it is proved that
  the image of any line is a normal rational curve
  in some subspace.
  In this note a more detailed geometric description is achieved.
  Consequences are found related to mixed partitions of the projective spaces;
  in particular, 
  it is proved that for any positive integer $k$, if $q\ge2^k-1$, then 
  there are partitions of $\PG(2^k-1,q)$ in normal rational curves
  of degree $2^k-1$. 
  For smaller $q$ the same construction gives partitions in $(q+1)$-tuples of independent points.\\
  {\sc A.M.S. Classification: 51E20 -- 51E23 -- 05B25.}
\end{abstract}

\section{Introduction}\begin{sloppypar}
Let $q$ be a power of a prime and $n$ an integer greater than one.
Let $L=\mathbb F_{q^n}$ be a finite field and let $F=\mathbb F_q$ be a subfield of $L$.
The projective space $\PG(n-1,q)$ is 
isomorphic to the projective space associated to the $F$-vector space $L$.
So, the points of $\PG(n-1,q)$ can be represented in the form $Fx$, $x\in L^*$.
Define the mapping $j:\,\PG(n-1,q)\rightarrow\PG(n-1,q)$ by 
\begin{equation}\label{defj}j(Fx)=Fx^{-1}.\end{equation}

In section \ref{proj-min-embedd} it will be shown that 
the map (\ref{defj}) appears spontaneously in the investigations on embeddings of minimum dimension
of the product spaces \cite{LaShZa13}.
It has been already studied in \cite{Fa02,Cs13}, where it is proved that if $q$ is large enough,
then the image of any line is a
normal rational curve of degree $m-1$, with $m$ dividing $n$.
Furthermore, the lines whose image is a line (so $m=2$ for even $n$) form a spread.
Some of the results in \cite{Fa02,Cs13} can be also derived from the description of affine chains given
in \cite{Ha83,He85}; see also \cite{BlHe05,Bu95}. 
The properties of $j$ which are relevant to this paper are described in section \ref{geometry-inverse}.

Section \ref{partitions} is devoted to the investigation on partitions of a finite projective space.
In \cite{Co00} mixed partitions of $\PG(3,q)$ consisting of two lines and $q^2-1$ normal rational curves
are found.
By the above results, the images under $j$ of 
the line spreads of $\PG(3,q)$ are mixed partitions containing
different numbers of lines. There are also partitions made only by normal rational curves.
A construction is described which can be extended to a partition in normal rational curves of any $\PG(2^k-1,q)$
(theorem\ \ref{x7}). \end{sloppypar}

\section{Projections in minimum embeddings}\label{proj-min-embedd}
\begin{sloppypar}
The study of the geometry of the inversion in a finite field is motivated by previous work \cite{LaShZa13}; 
it arises from
one of the projections in minimum embeddings of product spaces, as will be explained in this section.
We will assume that the reader is acquainted with the basic results and notation from  \cite{LaShZa13}, but 
skipping this
section should not affect the understanding of the following sections.
Consider the hypersurface $\calq_{n-1,q}$ of $\PG(2n-1,q)$ defined in \cite[(4,7,8)]{LaShZa13},
generalizing the semifield embeddings in $\PG(3,q)$ and $\PG(5,q)$ \cite{LaZa13}.
Assume that $h$ and $h'$ are integers such that $0<h'<h\le n-1$ and $\gcd(h-h',n)=1=\gcd(h',n)$, 
and take $k\in L$ with
$N(k)=1$.
For any point $P\in S_{h,k}$ \cite[(11)]{LaShZa13} precisely one subspace $S_{P,h,k}\in\cals_{h'}$ exists satisfying $P\in S_{P,h,k}$.
The correspondence $P\mapsto S_{P,h,k}$ is a bijection \cite[prop.\ 7]{LaShZa13}.
The map $\psi_{h,h',k}:\,S_{h,k}\rightarrow S_{0,1}$ defined by
$\psi_{h,h',k}(P)=S_{P,h,k}\cap S_{0,1}$ is a bijection and will be called the
\textit{projection} of indexes $h,h'$ and $k$.
\end{sloppypar}

In the following proposition, $\theta_s=(q^{s+1}-1)/(q-1)$ for any integer $s\ge-1$.
\begin{prop}
  Assume that $h$ and $h'$ are integers such that $0<h'<h\le n-1$ and $\gcd(h-h',n)=1=\gcd(h',n)$, 
  and that $k\in L$, $N(k)=1$.
  Then for any $x\in L^*$, it holds
    $\psi_{h,h',k}(F(x,kx^{q^h}))=F(t,t)$, where
   \begin{equation}\label{psi}
   t=\ell x^{\displaystyle q^{h'}\frac{1-q^{h-h'}}{q-1}e}, \quad \ell^{q-1}=k^{-e},\quad
   e\theta_{h'-1}\equiv1\mod \theta_{n-1}.
  \end{equation}
\end{prop}
\begin{proof}
  The definition of $\psi$ implies
  \begin{equation}\label{calcdiretto}
    t^{q^{h'}-1}=k^{-1}x^{q^{h'}-q^h}.
  \end{equation}
  Since $t$ is defined up to a non-zero factor in $F$, it is allowed
  to take $(q-1)$-th roots in (\ref{calcdiretto}), whence
  \begin{equation}\label{q-1roots}
    t^{\theta_{h'-1}}=\ell'x^{-q^{h'}\theta_{h-h'-1}},
  \end{equation}
  where $\ell'$ is such that $\ell'^{q-1}=k^{-1}$.
  The assumption $\gcd(h',n)=1$ implies $\gcd(\theta_{h'-1},\theta_{n-1})=1$,
  so integers $e$ and $\epsilon$ exist such that $e\theta_{h'-1}=\epsilon\theta_{n-1}+1$.
  By taking the $e(q-1)$-th power of (\ref{q-1roots}) and setting $\ell^{q-1}=k^{-e}$ one obtains
  \[
    t^{q-1}=\ell^{q-1}x^{\displaystyle q^{h'}(1-q^{h-h'})e}.
  \]
  Taking once again $(q-1)$-th roots gives (\ref{psi}).
\end{proof}
The easiest nontrivial case is $h'=1$, $h=2$, leading to the
composition of projectivities and $j$:
\begin{prop}
  Two collineations $S_{2,k}\stackrel{\alpha}{\rightarrow} \PG(n-1,q)\stackrel{\beta}{\rightarrow} S_{0,1}$ 
  exist such that
  $\psi_{2,1,k}=\alpha j\beta$.
\end{prop}
\begin{proof}
  Just take $\alpha:\,F(x,kx^{q^2})\mapsto Fx$ and 
  $\beta:\,Fx\mapsto F\ell(x^q,x^q)$.
\end{proof}

\section{Geometric properties}\label{geometry-inverse}

The mapping $j$ is well-defined and $j^2$ is the identity map in $\PG(n-1,q)$.
The one-dimensional $F$-subspace generated by $u\in F$ is $Fu$, but in cases like $F(u+v)$ the notation 
$\langle u+v\rangle_F$ will be preferred in order to avoid confusion with a simple algebraic extension of $F$.

In \cite[p. 13]{Ha97} the \textit{$n$-uple embedding} of an $r$-dimensional projective space $\PP^r$
over an algebraically closed field $K$ is defined to be the map
$\rho_d:\PP^r\rightarrow\PP^N$ defined by
  \[
    \rho_d(K(x_0,\ldots,x_r))=K(M_0(x_0,\ldots,x_r),\ldots,M_N(x_0,\ldots,x_r)),
  \]
where $N={r+d\choose d}-1$ and $M_0$, $M_1$, $\ldots$, $M_N$ are all homogeneous monomials of
degree $d$ in $x_0$, $x_1$, $\ldots$, $x_r$.

Here normal rational curves (NRCs for short) are defined in partial accord to \cite[Chapter 21]{JWPH2}.
It should be noted that they are \textbf{sets} of points.
In the following, $F$ is the finite field $\FF_q$.
A \textit{rational curve} $\calc_n$ of order $n$ in $\PG(r,q)$ is the set of points
\[
  \{F(g_0(t_0,t_1),\ldots,g_r(t_0,t_1))\mid (t_0,t_1)\in F^2\}
\]
where each $g_i$ is a binary form of degree $n$ and a highest common factor of $g_0$, $g_1$,
$\ldots$, $g_r$ is 1.
Also $\calc_n$ is \textit{normal} if it is not the projection of a rational curve $\calc_n'$ in
$\PG(r+1,q)$, where $\calc_n'$ is not contained in any $r$-subspace of $\PG(r+1,q)$.
Some issues arise from the fact that the order of a NRC is not unique, as is stated
in the next propositions.
\begin{prop}\label{117-1}
  The image of the $r$-uple embedding of $\PG(1,q)$:
  \begin{equation}\label{def_c}
    \calc=\{F(t_0^r,t_0^{r-1}t_1,\ldots ,t_1^r)\mid (t_0,t_1)\in F^2\setminus\{(0,0)\}\}
  \end{equation}
  is a NRC of order $r$ in $\PG(r,q)$.
\end{prop}
\begin{prop}\label{117-2}
  If $q<r$, then the curve $\calc$ in (\ref{def_c}) is a NRC of order $q$ in a $q$-subspace
  of $\PG(r,q)$.
\end{prop}
In the following, for any NRC the smallest order will be chosen.
Such order equals the dimension of its span \cite[Theorem 21.1.1]{JWPH2}.

\begin{teor}\label{main}
  Let $\ell$ be the line through two distinct points $Fa$, $Fb$ in $\PG(n-1,q)$.
  Then $j(\ell)$ is the image of an $(m-1)$-uple embedding
  with $m$ equal to the extension field degree of $F(ab^{-1})$, $m=[F(ab^{-1}):F]$.
\end{teor}
\begin{proof}
  The assertion can be deduced from \cite{Bu95}, as follows. 
  The projective space $\PG(n-1,q)$ can be considered as the hyperplane at infinity, say $H_\infty$, of the
  affine space $L$ with line set $\{Fc+d\mid c,d\in L,\,c\neq0\}$.
  The point at infinity of the line $Fc+d$ is $Fc$.
  The lines joining the origin and the points of $\ell'=Fa+b$ meet $H_\infty$ exactly in the points of 
  $\ell\setminus\{Fa\}$.
  By \cite[prop.\ 3.6.3, theorem 3.6.5]{Bu95},
  $j(\ell')\cup\{0\}$ is a NRC without points at infinity, say $\cal V$. 
  Furthermore, $j(Fa)$ is the tangent line in 0 to $\cal V$.
  The secant lines to $\cal V$ through 0 and the tangent line $j(Fa)$ meet $H_\infty$ in the image of an
  $(m-1)$-uple embedding. 
\end{proof}

\begin{prop}\label{131023}
  If $a,b\in L^*$ and $[F(ab^{-1}):F]=2$, then the restriction of $j$ to the line containing $Fa$ and $Fb$
  is a projectivity.
\end{prop}
\begin{proof}
  Let $\omega=ab^{-1}$ with $\omega^2=\alpha\omega+\beta$ and $\alpha,\beta\in F$.
  The assertion follows from $\langle (t+u\omega)^{-1}\rangle_F=\langle t+\alpha u-u\omega\rangle_F$ 
  for any $t,u\in F$.
\end{proof}

Now assume that $n$ is even.
A unique subfield $C\cong\FF_{q^2}$ exists in $L$.
Every $x\in L^*$ is associated with the line in $\PG(n-1,q)$ arising from the two-dimensional vector
subspace $Cx$. The symbol $Cx$ will denote the line itself.
The \textit{Desarguesian line spread} of $\PG(n-1,q)$ is $\cald=\{Cx\mid x\in L^*\}$.
\begin{prop}
  The restriction of $j$  to any line of $\cald$ is a projectivity with a line of $\cald$.
\end{prop}
\begin{proof}
  A line $\ell$ of $\cald$ contains two points of type $Fa$ and $Fb$ with $ab^{-1}\in C\setminus F$.
  So by prop.\ \ref{131023}, $j$ maps projectively $\ell$ into the line containing the points $Fa^{-1}$ and
  $Fb^{-1}$.
  The condition $ab^{-1}\in C$ implies that $j(\ell)\in\cald$.
\end{proof}
The following props. \ref{cs1} and \ref{cs2} as well as theorem \ref{spread-base} can be found in \cite{Cs13}:
\begin{prop}\label{cs1}
  Let $n$ be even.
  If $n\equiv 0 \mod 4$ and $q$ is odd, then
  the lines of $\cald$ which are fixed by $j$ are precisely two;
  otherwise there is a unique line which is fixed by $j$.
\end{prop}
\begin{proof}
  A line $Cx$ is fixed by $j$ if, and only if, $Cx=Cx^{-1}$, that is $x^2\in C$.
  An $x\in L\setminus C$ such that $x^2\in C$ exists if, and only if, $q$ is odd and
  $n\equiv 0 \mod 4$.
  Indeed, the divisibility of $n$ by 4 comes from the fact that $[C : F] = [C(x) : C] =2$, so that 
  $[C(x) : F] = 4$. Moreover, $q$ being odd is clear.
  If an $x\in L\setminus C$ such that $x^2\in C$ exists, there are at least two lines that are fixed by $j$, 
  precisely $C1$ and $Cx$.
  Otherwise only the line $C1$ is fixed by $j$.

  Now assume that $q$ is odd and $n\equiv 0 \mod 4$.
  Let $Cy$ and $Cz$ be any two lines fixed by $j$ and distinct from $C1$. 
  It holds $y,z\in L\setminus C$ and $y^2, z^2\in C$.
  Then $y^2,z^2$ are non-squares in $C$, and $y^2z^{-2}$ is a square in $C$.
  This implies $yz^{-1}\in C$ and finally $Cy=Cz$.
  So in this case there are at most two lines which are fixed by $j$.
\end{proof}
  
It should be noted that the image of the line through two distinct points $Fa$ and $Fb$
is a line if, and only if, $F(ab^{-1})=C$, as a consequence of the uniqueness of a subfield
of order $q^2$ in $L$.
This implies:
\begin{teor}\label{spread-base}
  The lines of $\PG(n-1,q)$ whose image under $j$ are lines are precisely those in $\cald$.
\end{teor}

The previous proposition can be partly generalized.
\begin{prop}\label{cs2}
  Assume $n$ is any integer and let $m>1$ be a divisor of $n$ and such that $m-1\ge q$.
  Then the $(m-1)$-subspaces of $\PG(n-1,q)$ whose images under $j$ are $(m-1)$-subspaces form a
  Desarguesian spread of $\PG(n-1,q)$.
\end{prop}

\section{Partitions in normal rational curves}\label{partitions}
\subsection{Three-dimensional case}
The term mixed partition in the literature refers to distinct geometric concepts.
Here a \textit{mixed partition of type $r$} in $\PG(3,q)$, $0\le r\le q^2$,
is a partition of $\PG(3,q)$ in $r$ lines and $q^2+1-r$ NRCs or order three.
In \cite{Co00} mixed partitions of type 2 are constructed and investigated.
The following is immediate by means of theorem \ref{spread-base}:
\begin{prop}\label{emp}
  If a line spread $\calf$ in $\PG(3,q)$ exists having exactly $r$ lines in common
  with a Desarguesian spread $\cald$, then there is a mixed partition of type $r$.
\end{prop}
By prop.\ \ref{emp}, in order to constuct a partition of $\PG(3,q)$ in NRCs
it is enough to find a line spread sharing no line with a given Desarguesian spread.
\begin{prop}\label{1112-1}
  The projective space $\PG(3,q)$ contains mixed partitions of type $0$, for any $q$.
\end{prop}
\begin{proof}
  Under the Klein correspondence, the spread $\cald$ is associated with the intersection
  of $Q^+(5,q)$ with a solid $S$, and $S\cap Q^+(5,q)$ is an elliptic quadric $\cale$.
  Take any line $\lambda$ in $S$ such that $\lambda\cap\cale=\emptyset$.
  The solid $\lambda^\perp$ intersects $\cale$ in two points, corresponding to
  two lines $\ell_1,\ell_2\in\cald$.
  The quadric $\lambda^\perp\cap Q^+(5,q)$ is associated with a Desarguesian spread $\calf_1$,
  and $\cald\cap\calf_1=\{\ell_1,\ell_2\}$.
  By replacing in $\calf_1$ a regulus containing $\ell_1$ and $\ell_2$ with the opposite one,
  a new spread $\calf$ is obtained having no common line with $\cald$.
  So, the assumptions of prop.\ \ref{emp} with $r=0$ are satisfied.
\end{proof}

\subsection{Generalization}\label{generalizzazione}
Now a generalization of prop.\ \ref{1112-1} to any odd-dimensional finite projective space is obtained.
In this subsection, $n>1$ is an integer, and
\[
  L=\mathbb F_{q^{2n}};\ M=\mathbb F_{q^n};\ C=\mathbb F_{q^2};\ F=\mathbb F_q;\ F\subset M\subset L;\ 
  F\subset C\subset L.
\]
Any of the fields above is considered as a vector space over $F$.
Let $i\in L\setminus M$, so any $z\in L$ can be described uniquely in the form $z=a+ib$ with $a,b\in M$.
The points of $\PG(2n-1,q)$ are $Fz$ for $z\in L^*$.

For every $a\in M$ define $S_a=\{Fc(a+i)\mid c\in M^*\}$, furthermore let
$S_{\infty}=\{Fc\mid c\in M^*\}$,  $M^+=M\cup\{\infty\}$, $\cals=\{S_a\mid a\in M^+\}$.
\begin{prop}\label{x2}
  \begin{enumerate}[{\rm (i)}]
    \item For any $a\in M^+$, $S_a$ is an $(n-1)$-subspace of $\PG(2n-1,q)$.
    \item If $a,a'\in M^+$, $a\neq a'$, then $S_a\cap S_{a'}=\emptyset$.
    \item The collection $\cals$ is the standard Desarguesian spread of $(n-1)$-subspaces of
    $\PG(2n-1,q)$.
  \end{enumerate}
\end{prop}
\begin{proof}
(i) The union of the one-dimensional subspaces belonging to $S_a$ is an $n$-subspace of $L$.

(ii) If $a\neq\infty$, then $\langle a+i\rangle_M\cap M=\{0\}$ implies $S_a\cap S_\infty=\emptyset$.
Now assume $a\neq\infty\neq a'$.
If $Fc(a+i)=Fc'(a'+i)$ for $c,c'\in M^*$, then $c(a+i)=\rho c'(a'+i)$ for some $\rho\in F$, and this
implies $a=a'$.

(iii) First, the union of the elements of any $S_a$, $a\in M^+$, is of type $Mz$, so it is an element
of the standard Desarguesian spread of $(n-1)$-subspaces of $\PG(2n-1,q)$.
Conversely, given an $Mu$ with $u=s+it\neq0$, $s,t\in M$, it holds
$Mu=\langle st^{-1}+i\rangle_M$ and the associated projective subspace is $S_{st^{-1}}$, possibly $S_\infty$.
\end{proof}

The reader will immediately verify the next result:
\begin{prop}\label{x3}
  \begin{enumerate}[{\rm (i)}]
    \item The map $\hat{\varphi}:\,L\rightarrow L$ defined by $\hat{\varphi}(a+ib)=a^q+ib$ for $a,b\in M$
    is an automorphism of the $F$-vector space $L$.
    \item Let $\varphi:\,\PG(2n-1,q)\rightarrow\PG(2n-1,q)$ be the projectivity associated with
    $\hat{\varphi}$.
    Then
    \[
      S_a^\varphi=\{\langle c^qa^q+ci\rangle_F\mid c\in M^*\}\mbox{ for }a\in M,\quad S_\infty^\varphi=S_\infty,
      \quad S_0^\varphi=S_0.
    \]
  \end{enumerate}
\end{prop}
  
Denote by $\cals^\varphi=\{S_a^\varphi\mid a\in M^+\}$ the image of the spread $\cals$ under $\varphi$.
\begin{prop}\label{x4}
  For any $a,b\in M^*$, the intersection $S_a\cap S_b^\varphi$ contains at most one point.
\end{prop}
\begin{proof}
  The common points to both subspaces are of type $\langle c(a+i)\rangle_F=\langle d^qb^q+di\rangle_F$ 
  for some $c,d\in M^*$.
  A $\rho\in F^*$ exists such that $ca=\rho d^qb^q$ and $c=\rho d$.
  This implies $d^{q-1}=ab^{-q}$.
  This equation has either zero or exactly $q-1$ roots in the form $\sigma d_0$, $\sigma\in F^*$.
\end{proof}

A \textit{scattered subset} with respect to a spread is a set of points intersecting any spread
element in at most one point.
\begin{teor}\label{x5}
  If $n$ is even, then a line spread $\calf$ of $\PG(2n-1,q)$ exists such that any line of $\calf$
  is scattered with respect to the Desarguesian spread $\cals$.
\end{teor}  
\begin{proof}
  Let $\calr$ be a regulus of $(n-1)$-subspaces such that $S_0,S_\infty\in\calr$, 
  and $\calr\subset\cals^\varphi$; such a regulus exists because $\cals^\varphi$ is a Desarguesian spread.
  Let $M^-=\{a\in M\mid S_a^\varphi\not\in\calr\}$.
  For any $a\in M^-$ choose a line spread $\calf_a$ of the $(n-1)$-subspace $S_a^\varphi$.
  Let $\calf'$ be the set of transversal lines of the regulus $\calr$.
  Define
  \[
    \calf=\calf'\cup\bigcup_{a\in M^-}\calf_a.
  \]
  By construction, $\calf$ is a line spread of $\PG(2n-1,q)$.
  Let $\ell$ be any line of $\calf$.
  If $\ell\in\calf_a$, $a\in M^-$, then $\ell$ is contained in a scattered subspace with respect to $\cals$
  (prop.\ \ref{x3}, \ref{x4}), so it is also scattered.
  If $\ell\in\calf'$, then it intersects at least two elements $S_0$ and $S_\infty$ of $\cals$ and this
  implies that $\ell$ is scattered.
\end{proof}

In the previous theorem, $n$ has to be even here in order to guarantee that each $(n-1)$-subspace 
admits a line spread.
Now we further assume $n=2^{k-1}$, $k>1$ an integer.
So,
\[
  L=\mathbb F_{q^{2^k}}\supset M=\mathbb F_{q^{2^{k-1}}}\supseteq C=\mathbb F_{q^2}\supset F=\mathbb F_q.
\]
\begin{teor}\label{x6}
  Let $\PG(2^k-1,q)=\{Fz\mid z\in L^*\}$.
  Then for any line $\ell$ of $\PG(2^k-1,q)$, $j(\ell)$ is the image of a $(2^k-1)$-uple
  embedding if, and only if, $\ell$ is scattered with respect to the Desarguesian spread $\cals$.
\end{teor}
\begin{proof}
  Given a line $\ell$ through two distinct points $Fz_1$, $Fz_2$, then $\ell$ is scattered if, and only if,
  $z_1z_2^{-1}\not\in M$.
  On the other hand, if $z_1z_2^{-1}\in M$, then, by theorem \ref{main}, $j(\ell)$ is a NRC of
  degree at most $2^{k-1}-1$, whereas if $z_1z_2^{-1}\not\in M$, since every proper subfield of $L$ is
  contained in $M$, 
  then  the extension $F(z_1z_2^{-1})$ is equal to $L$, and this implies once again by theorem \ref{main}
  that $j(\ell)$ is the image of a $(2^k-1)$-uple embedding.
\end{proof}

\begin{teor}\label{x7}
  For any $k>1$ and $q\ge2^k-1$, the projective space
  $\PG(2^k-1,q)$ has a partition in normal rational curves
  of degree $2^k-1$. 
  For $q<2^k-1$ the projective space has a partition in $(q+1)$-tuples of independent points.
\end{teor}
\begin{proof}
  This is a straightforward consequence of theorems \ref{x6} and \ref{x7}.
\end{proof}

\textbf{Acknowledgements}\quad
The authors thank Hans Havlicek for his helpful remarks in the preparation of this paper.
This research was supported by a Progetto di Ateneo from Universit\`a di Padova (CPDA113797/11).



\begin{thebibliography}{pippo}

\bibitem{BlHe05}
{\sc A. Blunck - A. Herzer:} \textit{Kettengeometrien -- Eine Einf{\"u}hrung.}
Shaker Verlag, 2005, Aachen. 

\bibitem{BoCo00}
{\sc A. Bonisoli - A. Cossidente:} Mixed partitions of projective geometries. 
Des.\ Codes Cryptogr.\ {\bf20} (2000), 143--154.

\bibitem{Co00}
{\sc A. Cossidente:} Mixed partitions of $\PG(3,q)$.
J. Geom.\ {\bf68} (2000), 48--57.

\bibitem{CoFuLa01}
{\sc A. Cossidente - M. Funk - D. Labbate:} Graphs arising from mixed partitions of $\PG(5,q)$. 
Bull.\ Inst.\ Combin.\ Appl.\ {\bf33} (2001), 57--67.

\bibitem{CoSo10}
{\sc A. Cossidente - A. Sonnino:} Finite geometry and the Gale transform.
Discrete Math.\ {\bf310} (2010), 3206--3210.

\bibitem{Cs13}
{\sc B. Csajb\'ok:} Linear subspaces of finite fields with large inverse-closed subsets.
Finite Fields Appl.\ {\bf19} (2013), 55--66. 

\bibitem{Fa02}
{\sc G. Faina - G. Kiss - S. Marcugini - F. Pambianco:} The cyclic model for $\PG(n,q)$ and a construction
of arcs.
European J. Combin.\ {\bf23} (2002), 31--35.


\bibitem{Ha97}
{\sc R. Hartshorne:} \textit{Algebraic geometry.} Springer, New York, 1977.


\bibitem{Ha83}
{\sc H. Havlicek:} Eine affine Beschreibung von Ketten. Abh.\ Math.\ Sem.\ Univ.\ Hamburg 
{\bf53} (1983), 266--275.

\bibitem{He85}
{\sc A. Herzer:} \"Uber die Darstellung affiner Ketten als Normkurven. Abh.\ Math.\ Sem.\ Univ.\ Hamburg 
{\bf55} (1985), 211--228.

\bibitem{Bu95}
{\sc A. Herzer:} Chain geometries. In F. Buekenhout, editor, \textit{Handbook of incidence geometry. Buildings and foundations.}
Elsevier, Amsterdam, 1995.

\bibitem{JWPH2}
{\sc J.W.P. Hirschfeld:} \textit{Finite projective spaces of three dimensions.}
Oxford University Press, 1985.

\bibitem{JWPH3}
{\sc J.W.P. Hirschfeld - J.A.Thas:} \textit{General Galois geometries.}
Clarendon Press, 1991.







\bibitem{LaShZa13}
{\sc M. Lavrauw - J. Sheekey - C. Zanella:} On embeddings of minimum dimension of $\PG(n,q)\times\PG(n,q)$.
Des.\ Codes Cryptogr., to appear. DOI 10.1007/s10623-013-9866-8

\bibitem{LaZa13}
{\sc M. Lavrauw - C. Zanella:} Segre embeddings and finite semifields.
Finite Fields Appl. {\bf25} (2014), 8--18.




\end{thebibliography}
\end{document}